\documentclass[12pt]{amsart}
\usepackage{amssymb, latexsym,amsmath,amsfonts}

\def\be{\begin{enumerate}}

\def\ee{\end{enumerate}}

\theoremstyle{plain}

\newtheorem{thm}{Theorem}[section]

\newtheorem{cor}[thm]{Corollary}

\newtheorem{lemma}[thm]{Lemma}

\theoremstyle{definition}

\numberwithin{equation}{section}

\title[finite homogene ous...]{Finite homogeneous effect algebras with trivial sharp elements}

\begin{document}

\author[G. Bi\'nczak]{G. Bi\'nczak$^1$}

\address{$^1$ Faculty of Mathematics and Information Sciences\\
Warsaw University of Technology\\
00-662 Warsaw, Poland}

\author[J. Kaleta]{J. Kaleta$^2$}

\address{$^2$ Department of Applied Mathematics\\
 Warsaw University of Agriculture\\02-787 Warsaw, Poland}

\email{$^1$ binczak@mini.pw.edu.pl, $^2$joanna\_kaleta@sggw.pl}

\keywords{effect algebras, sharp elements, homogeneous effect algebras}

\subjclass[2010]{81P10, 81P15}

\begin{abstract}
In this paper we describe  finite homogeneous effect algebras whose sharp elements are only $0$ and $1$.
\end{abstract}

\maketitle

\section{\bf Introduction}
Effect algebras have been introduced by Foulis and Bennet in 1994 (see \cite{FB94}) for the study of foundations of quantum mechanics (see \cite{DP00}). Independtly, Chovanec and K\^opka introduced an essentially equivalent structure called $D$-{\em poset} 
(see \cite{KC94}). Another equivalent structure was introduced by Giuntini and Greuling in \cite{GG94}).

In this paper we partially  solved (see Corollary \ref{cor3})  the following  Open Problem (see \cite{RA15}):

Problem 10.1 (G. Jen\v ca). Prove or disprove: if $E$ is an orthocomplete  homogeneous effect algebra $E$ such that $S(E)$ is a lattice, then $ E$ is a lattice effect algebra.

Let $E$ be an effect algebra, $x,y\in E$ then the set $\{z\in E\colon x\leq z\leq y\}$ we denote by $[x,y]$.

\begin{lemma}\label{lem22}
Let $E$ be a homogeneous effect algebra, $a\in E$ be an atom such that $a\leq a'$. If $v_1,v_2\in E$, $v_1\perp v_2$ and $v_1\oplus v_2\in[a,a']$ then $a\leq v_1$ or $a\leq v_2$.
\end{lemma}

\begin{proof}
Let $a\leq v_1\oplus v_2\leq a'$. Then there exist $u_1,u_2\in E$ such that $u_1\leq v_1$, $u_2\leq v_2$ and $u_1\oplus u_2=a$ since
$E$ is homogeneous. Moreover $0\leq u_1\leq u_1\oplus u_2=a$ and $0\leq u_2\leq u_1\oplus u_2=a$, so $u_1,u_2\in\{0,a\}$ since $a$ is an atom. We know that $0\oplus0=0\not=a$ and $a\oplus a\not=a$, hence ($u_1=0$ and $u_2=a$ ) or ($u_2=0$ and $u_1=a$). Therefore $a=u_1\leq v_1$ or $a=u_2\leq v_2$.
\end{proof}

\begin{lemma}\label{lem14}
Let $E$ be an effect algebra and $x\in E$ then $x\notin S(E)\iff$ there exists $b\in E$ such than $b\not=0$ , $b\leq b'$ and $x\in[b,b']$.
\end{lemma}

\begin{proof}
If $x\notin S(E)$ then $0$ is not the greatest lower bound of the set $\{x,x'\}$, so there exists $b\not=0$ such that $b\leq x$ and $b\leq x'$ hence $b\leq x\leq b'$. Therefore $b\not=0$, $b\leq b'$ and $x\in[b,b']$.

If there exist  $b\in E$ such than $b\not=0$ , $b\leq b'$ and $x\in[b,b']$ then $b\leq x\leq b'$ so $b\leq x$ and $b\leq x'$ and $b\not=0$
hence $x\land x'\not=0$ and $x\notin S(E)$.

\end{proof}

\begin{lemma}\label{lem15}
Let $E$ be an effect algebra and $x\in E$. If $x$ is an atom and $x\notin S(E)$ then $x\leq x'$.
\end{lemma}
\begin{proof}
Let $x\in E$ be an atom and $x\notin S(E)$. By Lemma\ref{lem14} there exist $b\in E$ such that $b\not=0$, $b\leq b'$ and $x\in [b,b']$.
so $0\leq b\leq x$ hence $b=x$ (since $x$ is an atom and $b\not=0$) so $x\leq x'$.
\end{proof}

\begin{lemma}\label{lem20}
Let $E$ be a finite effect algebra such that $S(E)=\{0,1\}$ then
$$E=\{0,1\}\cup\bigcup\limits_{n=1}^k[a_n,a_n'],$$
where $\{a_1,\ldots a_k\}$ is the set of all atoms of $E$.
\end{lemma}
\begin{proof}
By Lemma \ref{lem15} $a_i\leq a_i'$ for $i=1,2,\ldots,k$. Let $x\in E$ and $x\notin\{0,1\}$ then $x\notin S(E)$. By Lemma\ref{lem14}
there exist $b\in E$ such that $b\not=0$ and $x\in[b,b']$. There exist $i\in\{1,\ldots,k\}$ such that $a_i\leq b$ since $E$ is finite.
Hence
$$a_i\leq b\leq x\leq b'\leq a_i'$$
so $x\in[a_i,a_i']$.
\end{proof}

\begin{lemma}\label{lem30}
Let $E$ be a finite homogeneous effect algebra such that $S(E)=\{0,1\}$. Let $\{a_1,\ldots,a_k\}$ be the set of all atoms of $E$.

If $na_i=\underbrace{a_i\oplus\ldots\oplus a_i}_{n-times}$ is well defined ($n\geq1$) and $na_i\in[a_j,a_j']$ then $i=j$.
\end{lemma}

\begin{proof}
(by induction with respect to $n$): If $1a_i=a_i\in[a_j,a_j']$ then $0\leq a_j\leq a_i$ so $a_i=a_j$ since $a_i$ is an atom. Therefore $i=j$ and this ends proof for $n=1$.

Suppose that this lemma is true for $n$,  $(n+1)a_i$ is well defined and $(n+1)a_i=na_i\oplus a_i\in[a_j,a_j']$. By Lemma~\ref{lem22}
$a_j\leq a_i$ or $a_j\leq na_i$.

 If $0\leq a_j\leq a_i$ then $a_i=a_j$ since $a_i$ is an atom, so $i=j$.

If $a_j\leq na_i$ then $a_j\leq na_i\leq(n+1)a_i\leq a_j'$ so $na_i\in[a_j,a_j']$ and $i=j$ since lemma is true for $n$. Therefore
lemma holds for $n+1$.
\end{proof}

\begin{lemma}\label{lem31}
Let $E$ be a finite homogeneous effect algebra such that $S(E)=\{0,1\}$. Let $\{a_1,\ldots,a_k\}$ be the set of all atoms of $E$.
If $n\geq1$ and $na_i$ is well defined then $na_i=1$ or $na_i\in[a_i,a_i']$.
\end{lemma}
\begin{proof}
If $na_i$ is well defined and $na_i\not=1$ then $na_i\geq a_i>0$ so $na_i\notin\{0,1\}=S(E)$ and by Lemma\ref{lem20}
$na_i\in[a_j,a_j']$ for some $j\in\{1,\ldots,k\}$ by Lemma\ref{lem30} $i=j$ so $na_i\in[a_i,a_i']$.
\end{proof}

\begin{lemma}\label{lem32}
Let $E$ be a finite homogeneous effect algebra such that $S(E)=\{0,1\}$. Let $\{a_1,\ldots,a_k\}$ be the set of all atoms of $E$.
For every $i\in\{1,\ldots,k\}$ there exists $n\in N$ such that $a_i'=na_i$
\end{lemma}
\begin{proof}
Let $A=\{x\in E\colon \exists_{n\in N,n\geq1}na_i$ is well defined and $x=na_i\}$. Then $A$ is not empty since $a_i\in A$.
Let $x_0=na_i$ be a maximal element in $A$. By Lemma\ref{lem31} $na_i=1$ or $na_i\in[a_i,a_i']$.

If $na_i\in[a_i,a_i']$ then $na_i\leq a_i'$ so $(n+1)a_i$ is well defined and $x_0=na_i<(n+1)a_i\in A$ (since $n\geq 1$) and we obtain a contradiction with maximality of $x_0$. Hence $na_i=1$ and $(n-1)a_i\oplus a_i=na_i=1$ so $a_i'=(n-1)a_i$.
\end{proof}

\begin{lemma}\label{lem33}
Let $E$ be a finite homogeneous effect algebra such that $S(E)=\{0,1\}$. Let $\{a_1,\ldots,a_k\}$ be the set of all atoms of $E$.
Let $1\leq i\leq k$ and suppose that $na_i$ is well defined. If $[0,na_i]\not=\{0,a_i,2a_i,\ldots,na_i\}$ then there exists $j\in\{1,\ldots,k\}$ such that
$a_j\in[0,na_i]$ and $a_i\not=a_j$.
\end{lemma}

\begin{proof}
(by induction with respect to $n$): If $n=1$ then $[0.a_i]=\{0,a_i\}$ since $a_i$ is an atom.

Suppose that this lemma is true for $n$ and $[0,(n+1)a_i]\not=\{0,a_i,\ldots,(n+1)a_i\}$.

 If $[0,na_i]\not=\{0,a_i,2a_i,\ldots,na_i\}$ then (since lemma is true for $n$) there exists $j\in\{1,\ldots,k\}$ such that $a_j\in[0,na_i]$ and $a_i\not=a_j$ so $a_j\leq na_i\leq (n+1)a_i$ and $a_j\in[0,(n+1)a_i]$ hence lemma is true for $n+1$.

 If $[0,na_i]=\{0,a_i,2a_i,\ldots,na_i\}$ then every minimal element of $A=[0,(n+1)a_i]\setminus\{0,a_i,\ldots,(n+1)a_i\}$ is an atom:

Let $x_0$ be a minimal element in $A$, $0\leq y\leq x_0$ and $y\not=0$. We prove that $y=x_0$:

If $a_i\leq x_0$ then $x_0\in[a_i,(n+1)a_i]$ but there is an order isomorphism between $[0,na_i]$ and $[a_i,(n+1)a_i]$ ($z\mapsto z\oplus a_i$) and  $[0,na_i]=\{0,a_i,2a_i,\ldots,na_i\}$ so $x\in[a_i,(n+1)a_i]=\{a_i,\ldots,(n+1)a_i\}$ hence $x_0\notin A$ and we obtain a contradiction.

Therefore $a_i\nleq x_0$. We show that $y\in A$. We have $y\in[0,(n+1)a_i]$ since $y\leq x\leq (n+1)a_i$ and $y\not=ma_i$ for
$m=1,\ldots,n+1$ because if $y=ma_i$ then $a_i\leq ma_i=y\leq x_0$ a contradiction since $a_i\nleq x_0$ hence $y\in A$.
Therefore $y=x_0$ since $x_0$ is a minimal element of $A$. Hence $x_0$ is an atom and there exists $j\in\{1,\ldots,k\}$ such that $x_0=a_j$ and $a_j=x_0\not=a_i$ since $x_0\in A$
\end{proof}

\section{Main Theorem}

\begin{thm}\label{thm36}
Let $E$ be a finite homogeneous effect algebra such that $S(E)=\{0,1\}$. Let $\{a_1,\ldots,a_k\}$ be the set of all atoms of $E$.
Let $1\leq i\leq k$ and suppose that $na_i$ is well defined. If $na_i\leq a_i'$ then $[0,na_i]=\{0,a_i,2a_i,\ldots, na_i\}$.
\end{thm}
\begin{proof}
Assume that $[0,na_i]=\{0,a_i,2a_i,\ldots, na_i\}$. By Lemma \ref{lem33}  there exists $j\in\{1,\ldots,k\}$ such that
$a_j\in[0,na_i]$ and $a_i\not=a_j$. Let $A=\{(r,w)\in N^2\colon ra_i\oplus wa_j$ is well defined $\}$.

We know that $a_j\leq na_i\leq a_i'$ so $a_j\leq a_i'$ and $a_j\perp a_i$ hence $a_i\oplus a_j$ is well defined and $(1,1)\in A$ so
$A$ is not empty and is finite (since $E$ is finite). Then there exists  a maximal element   $(r_0,w_0)$ in $A$ such that $r_0\geq 1$ and $w_0\geq 1$. Hence $r_0a_i\oplus w_0a_j$ is well defined,  $r_0a_i\oplus w_0a_j\nleq a_i'$ and $r_0a_i\oplus w_0a_j\nleq a_j'$. Now we consider two cases
\be
	\item $r_0a_i\oplus w_0a_j=1$.  By Lemma \ref{lem32} there exists $w_1\in N$ such that $w_1a_j=a_j'$ and $(w_1+1)a_j=1$ so
	$wa_j$ is not well defined for $w>w_1+1$ but $w_0a_j$ is well defined hence $w_0\leq w_1+1$. Consider 2 subcases:
	\be
		\item $w+0=w_1+1$ then $w_0a_j=1$ and $r_0a_1\oplus 1$ is well defined so $r_0a_i=0$ and $a_i=0$ a contradiction since 				$a_i$ is an atom.
		\item $w_0\leq w_1$  by Lemma \ref{lem32} there exists $r_1\in N$ such that $r_1a_j=a_j'$ and $(r_1+1)a_i=1$.
			 If $r_0a_i=1$ then $1\oplus w_0a_j$ is well defined (since $r_0a_i\oplus w_0a_j$ is well defined) so $w_0a_j=0$ and 					$a_j=0$ a contradiction since $a_i$ is an atom. Hence $r_0a_i\not=1$ and $r_0\leq r_1$ thus $r_0a_i\leq r_1 a_i=a_i'$.
			We know that $r_0a_i\oplus w_0a_j=1$ so $r_0a_i=(w_0a_j)'=(w_1+1-w_0)a_j\in[a_i,a_i']$ and $w_1+1-w_0\geq 1$
			since $w_0\leq w_1$. By Lemma \ref{lem30} $i=j$ and $a_i=a_j$. We obtain a  a contradiction since $a_i\not=a_j$.
	\ee
	\item $r_0a_i\oplus w_0a_j\not=1$. By Lemma \ref{lem20} there exists $s\in\{1,\ldots,k\}$ such that 								$r_0a_i\oplus w_0a_j\in[a_s,a_s']$. By Lemma \ref{lem22} $a_s\leq r_0a_i$ or $a_s\leq w_0a_j$.
		\be
			\item If $a_s\leq r_0a_i$ then $a_s\leq r_0a_i\leq r_0a_0\oplus w_0a_j\leq a_s'$ and $r_0a_i\in[a_s,a_s']$.
				By Lemma \ref{lem30} $s=i$ hence $r_0a_i \oplus w_0a_j\leq a_i'$ a contradintion since 
				$r_0a_i \oplus w_0a_j\nleq a_i'$.
			\item If $a_s\leq w_0a_j$ then $a_s\leq w_0a_j\leq r_0a_0\oplus w_0a_j\leq a_s'$ and $w_0a_j\in[a_s,a_s']$.
				By Lemma \ref{lem30} $s=j$ hence $r_0a_i \oplus w_0a_j\leq a_j'$ a contradintion since 
				$r_0a_i \oplus w_0a_j\nleq a_j'$.
		\ee
\ee
Hence 
$[0,na_i]=\{0,a_i,2a_i,\ldots, na_i\}$.
\end{proof}
\begin{cor}\label{cor1}
Let $E$ be a finite homogeneous effect algebra such that $S(E)=\{0,1\}$. Let $\{a_1,\ldots,a_k\}$ be the set of all atoms of $E$.
If $i,j\in\{1,\ldots,k\}$, $x\in[a_i,a_i']$, $y\in[a_j,a_j']$ and $i\not=j$ then $x\nleq y$ and $[a_i,a_i']\cap[a_j,a_j']=\varnothing$.
\end{cor}

\begin{proof}
By Lemma \ref{lem32} there exist $n,m\in N$ such that $a_i'=na_i$ and $a_j'=ma_j$. Assume that $z\in[a_i,a_i']\cap[a_j,a_j']$ then
$z\in[0,na_i]$ so by Theorem \ref{thm36} $z=ra_i$ for some $r\in\{0,\ldots,n\}$ hence $ra_i\in[a_j,a_j']$ and by Lemma \ref{lem30}
we have $i=j$. Therefore $[a_i,a_i']\cap[a_j,a_j']=\varnothing$ for $i\not=j$.

Suppose that $x\in[a_i,a_i']$, $y\in[a_j,a_j']$ and $x\leq y$. Then $a_i\leq x\leq y\leq a_j'=ma_j$ so $a_i\in[0,ma_j]$. By Theorem \ref{thm36} there exists $s\in\{0,\ldots,m\}$ such that $a_i=sa_j$. If $s=0$ then $a_i=0$ a contradiction since $a_i$ is an atom.
If $s\geq1$ then $a_i=sa_j\geq a_j\geq0$ so $a_i=a_j$ since $a_i$ is an atom. Thus $i=j$.
\end{proof}

\begin{cor}\label{cor2}
Let $E$ be a finite homogeneous effect algebra such that $S(E)=\{0,1\}$. Then $E$ is isomorphic to horizontal sum of chains.
\end{cor}
\begin{proof}
 Let $\{a_1,\ldots,a_k\}$ be the set of all atoms of $E$. By Theorem \ref{thm36} $[a_i,a_i']$ is a chain. By Corollary \ref{cor1} $[a_i,a_i']\cap[a_j,a_j']=\varnothing$ for $i\not=j$ and  if $x\in[a_i,a_i']$, $y\in[a_j,a_j']$ and $i\not=j$ then $x\nleq y$ so by Lemma \ref{lem14} $E=\{0,1\}\cup\bigcup_{n=1}^k[a_i,a_i']$ is isomorphic to horizontal sum of chains.
\end{proof}

\begin{cor}\label{cor3}
Let $E$ be a finite homogeneous effect algebra. If  $S(E)=\{0,1\}$ then $E$ is a lattice.
\end{cor}
\begin{proof}
By Corollary \ref{cor2} $E$ is isomorphic to horizontal sum of chains so $E$ is a lattice.
\end{proof}

\end{document}